\documentclass[12pt]{amsart}
\usepackage{ amsmath, amsthm, amsfonts, amssymb, color}
 \usepackage{mathrsfs}
\usepackage{amsfonts, amsmath}
 \usepackage{amsmath,amstext,amsthm,amssymb,amsxtra}
 \usepackage{txfonts} 
 \usepackage[colorlinks, citecolor=blue,pagebackref,hypertexnames=false]{hyperref}
\usepackage{cite} 
\allowdisplaybreaks
 \usepackage{pgf,tikz}
 \usepackage{multirow}
 \usepackage{diagbox} 

 \usepackage{tikz}

\usetikzlibrary{decorations.pathreplacing}

\begin{document}

\baselineskip 16.6pt
\hfuzz=6pt

\widowpenalty=10000

\newtheorem{cl}{Claim}
\newtheorem{theorem}{Theorem}[section]
\newtheorem{proposition}[theorem]{Proposition}
\newtheorem{coro}[theorem]{Corollary}
\newtheorem{lemma}[theorem]{Lemma}
\newtheorem{definition}[theorem]{Definition}
\newtheorem{assum}{Assumption}[section]
\newtheorem{example}[theorem]{Example}
\newtheorem{remark}[theorem]{Remark}
\renewcommand{\theequation}
{\thesection.\arabic{equation}}

\def\SL{\sqrt H}

\newcommand{\cent}{\operatorname{cent}}

\newcommand{\wid}{\operatorname{width}}
\newcommand{\heit}{\operatorname{height}}
\newcommand{\intav}{-\!\!\!\!\!\!\int}

\newcommand{\cdim}{n}

\newcommand{\set}[1]{\mathfrak{#1}}
\newcommand{\vrect}{\set{V}}
\newcommand{\Stack}{\set{S}}
\newcommand{\tile}{\set{T}}

\newcommand{\mar}[1]{{\marginpar{\sffamily{\scriptsize
        #1}}}}

\newcommand{\as}[1]{{\mar{AS:#1}}}
\newcommand\C{\mathbb{C}}
\newcommand\Z{\mathbb{Z}}
\newcommand\R{\mathbb{R}}
\newcommand\RR{\mathbb{R}}
\newcommand\CC{\mathbb{C}}
\newcommand\NN{\mathbb{N}}
\newcommand\ZZ{\mathbb{Z}}
\newcommand\HH{\mathbb{H}}
\def\RN {\mathbb{R}^n}
\def\ls{\lesssim}
\renewcommand\Re{\operatorname{Re}}
\renewcommand\Im{\operatorname{Im}}

\newcommand{\mc}{\mathcal}
\newcommand\D{\mathcal{D}}
\def\hs{\hspace{0.33cm}}
\newcommand{\la}{\alpha}
\def \l {\alpha}
\newcommand{\eps}{\tau}
\newcommand{\pl}{\partial}
\newcommand{\supp}{{\rm supp}{\hspace{.05cm}}}
\newcommand{\x}{\times}
\newcommand{\lag}{\langle}
\newcommand{\rag}{\rangle}

\newcommand{\lset}{\left\lbrace}
\newcommand{\rset}{\right\rbrace}

\newcommand\wrt{\,{\rm d}}

\title[]{Commutators of maximal operators on weighted Morrey spaces}

\author{Manasa N. Vempati}
\address{}
\email{}

  \date{}

\begin{abstract}
In this article we obtain the characterization for the commutators of maximal functions on the weighted Morrey spaces in the setting of spaces of homogeneous type. More precisely, we characterize BMO spaces using the commutators of Hardy-Littlewood maximal function, sharp and fractional maximal functions.


\end{abstract}

\subjclass[2010]{42B25, 43A85}
\keywords{Commutator, maximal function, space of homogeneous type, BMO space, weight, Morrey spaces}

\maketitle

\section{\bf Introduction and Statement of Main Results}
\setcounter{equation}{0}

The study of maximal operator and their commutators, including their boundedness characterizations on certain function spaces, is a central topic in harmonic analysis and non-linear PDE's. In 2000, Bastero, Milman and Ruiz \cite{jmf} studied the necessary and sufficient condition for the boundedness of commutator of the Hardy–Littlewood maximal function and the commutator of the sharp maximal function on $L^p(\mathbb{R}^n)$. Later,  Zhang and Wu \cite{zlw}, obtained a similar result for the commutator of the fractional maximal function. Since then, the boundedness of these classical operators and their commutators have been extensively studied on various function spaces in the Euclidean setting, such as the weighted $L^p$ spaces (e.g. \cite{dca, dsa}) and Morrey spaces (e.g. \cite{a, ss, bg1, bg2, gcm}).

Motivated by these results, in this article we provide the boundedness characterizations for maximal function commutators on the Morrey spaces (weighted) in the general setting of spaces of homogeneous type. In modern harmonic analysis, a central focus has been extending the real-variable theory beyond the Euclidean framework, where the underlying space is $\mathbb{R}^n$, equipped with the Euclidean metric and Lebesgue measure, to more general settings such as the spaces of homoegenous type. These spaces have been extensive studied; see, for e.g., \cite{clw, nvm2, rg1, rg2, vnm1, djs, dh, cmlv, cfiv}.

 Consider $(X,d,\mu)$, a space of homogeneous type and let $p\in (1,\infty)$ and $\kappa\in (0,1)$. The weighted Morrey space $L_{\omega, \nu}^{p,\kappa}(X)$ corresponding to two weights $\omega$ and $\nu$ on $X$, in the sense of Komori and Shirai \cite{ks} is defined as

\vspace{-0.1 cm}

\begin{equation*}
   L_{\omega,\nu}^{p,\kappa}(X) := \{f\in L_{loc}^{p}(X,\nu):\|f\|_{L_{\omega,\nu}^{p,\kappa}(X)} <\infty\},
\end{equation*}
where
\begin{equation*}
    \|f\|_{L_{\omega,\nu}^{p,\kappa}(X)} := \sup_{B}\left\{\frac{1}{[\nu(B)]^\kappa}\int_{B}|f(h)|^p \omega(h)d\mu(h)\right\}^\frac{1}{p}.
\end{equation*}
\vspace{0.1 cm}

and the supremum is taken over all balls $B$ in $X$. If $\omega = \nu,$ then we denote by $L_{\omega}^{p,\kappa}(X)$, the classical weighted Morrey space. Note that, when $\kappa=0,$ then $L_{\omega, \nu}^{p,\kappa}(X)= L^p(X,\omega),$ the weighted $L^p$ space on $X.$

Morrey spaces, alongside weighted $L^p(w)$ spaces, play a pivotal role in the study of partial differential equations. In particular, Morrey spaces have proven highly useful for studying the local behavior of solutions to elliptic partial differential equations and have been heavily researched (see, for example \cite{gvqx,jm, akl, rjem, ssht, vsg}). Since the weighted Morrey spaces naturally generalize weighted Lebesgue spaces, thus, along this literature, it is natural to study the boundedness of maximal operators commutators on these spaces.

\vspace{0.1 cm}
We now recall the definition of the the Hardy--Littlewood maximal function on the space of homogeneous type $(X,d,\mu)$. For a locally integrable function $f$ on $X$, the Hardy--Littlewood maximal function $ M_pf(x)$, $1\leq p<\infty$ is defined as
$$  M_pf(x):=\sup_{B \ni x} \left({\frac{1}{\mu(B)}}\int_B |f(y)|^p\,d\mu(y)\right)^{\frac{1}{p}}, $$
where the supremum is taken over all balls $B\subset X$. Observe that, for $p=1$, $M=M_1$ is the classical Hardy–Littlewood maximal operator. 

For a fixed ball $B_0 \in X$, we denote by $M_{p, B_0}f(x)$, the Hardy–Littlewood maximal function relative to $B_0$,

\vspace{-0.5 cm}

\begin{equation*}
 M_{p, B_0}f(x):=\sup_{x\in B \subset B_0} \left(\frac{1}{\mu(B)}\int_B |f(y)|^p\,d\mu(y)\right)^{\frac{1}{p}},
\end{equation*}

Now we recall that, the commutator of the Hardy–Littlewood maximal operator given $b\in L^1_{loc}(X)$, is defined for all $x\in X$ as
\vspace{-0.5 cm}

\begin{equation*}
    [M_p, b]f(x) :=  b(x)M_pf(x) - M_P(bf)(x).
\end{equation*}

Below, we state our result concerning the boundedness characterization for the commutator of the Hardy--Littlewood maximal functions.

\begin{theorem}\label{thm main1}
Let $\kappa\in(0,1), \omega\in A_p(X)$ for $1<p<\infty$ and $b$ be a real valued, locally integrable function in $(X,d,\mu).$ Then the following conditions are equivalent:

\begin{enumerate}
\item[(i)] If $b\in BMO(X)$ and $b^-\in L^\infty(X),$
\item[(ii)] The commutator $[M_p, b]$ is bounded on $L_{\omega}^{q,\kappa}(X)$, $p<  q < \infty$.
\item[(iii)] For some $q\in (1,\infty),$ there exists some constant $C$, such that
    \begin{equation}
    \sup_{B}  \frac{\|(b - M_{p,B}(b))\chi_B\|_{L_{\omega}^{q,\kappa}(X)}}{\|\chi_B\|_{L_{\omega}^{q,\kappa}(X)}}  < C.
    \end{equation}
\end{enumerate}
\end{theorem}

For a locally integrable function  $b(x)$, we denote by $C_b$, the maximal commutator on $X$,
$$C_b(f)(x):=\sup_{B\ni x}{\frac{1}{\mu(B)}}\int_{B}|b(x)-b(y)| |f(y)|d\mu(y),$$
where the supremum is taken over all balls $B\subset X$. Note that, from \cite{mgkm}, we have the following upperbound for the maximal commutator $C_b$

\vspace{-0.5 cm}

\begin{equation}\label{maxcommu}
    C_b(f)(x) \leq C\|b\|_{BMO(X)}M^2(f)(x).
\end{equation}

As a consequence of the Theorem \ref{thm main1} and the equation \ref{maxcommu} we obtain the following boundedness characterization for the maximal commutator.

\begin{theorem}\label{thmax}
    Let $b$ be a locally integrable function on $(X,d,\mu)$, and $1<p<\infty$ and $\omega \in A_p(X)$. Then the maximal commutator $C_b$ is bounded on $L_{\omega}^{p,\kappa}(X)$ if and only if $b\in BMO(X).$
\end{theorem}

For any $f\in L^1_{loc}(X)$ and $x\in X,$ the sharp maximal function of Fefferman and Stein, $M^{\sharp}f$ is given as
\vspace{-0.5 cm}

\begin{equation*}
  M^{\sharp}f(x) =   \sup_{B \ni x} {1\over \mu(B)}\int_B |f(y) - f_B|\,d\mu(y),
\end{equation*}

where $f_B = \frac{1}{\mu(B)}\int_{B}f(y) du(y)$, the average value of $f$ on $B$. And for some $\gamma \in (0,1)$, we define the fractional maximal function $M_{\gamma}$ for $f\in L^1_{loc}(X)$ and $x\in X,$ as
\vspace{-0.3 cm}

\begin{equation*}
   M_{\gamma}f(x):=\sup_{B \ni x} \frac{1}{\mu(B)^{1-\gamma}}\int_B |f(y)|\,d\mu(y).
\end{equation*}

For any fixed ball $B_0 \in X$, we will denote by $M_{\gamma, B_0}f(x)$, the fractional maximal function relative to $B_0$. Now, we state our boundedness results for the commutators of sharp maximal functions and fractional maximal functions.




\vspace{0.1 cm}



\begin{theorem}\label{thm main2}
Let $\kappa\in(0,1), \omega\in A_{1}(X)$ and $b$ be a real valued, locally integrable function in $(X,d,\mu).$ Then the following conditions are equivalent:

\begin{enumerate}
 \item[(i)] If $b\in BMO(X)$ and $b^-\in L^\infty(X),$
\item[(ii)] The commutator $[ M^{\sharp}, b]$ is bounded on $L_{\omega}^{q,\kappa}(X)$, $1< q < \infty$.
\item[(iii)] For some $q\in (1,\infty),$ there exists some constant $C$, such that
    \begin{equation}
        \sup_{B}\frac{\|(b - 2M^{\sharp}(b))\chi_B\|_{L_{\omega}^{q,\kappa}(X)}}{\|\chi_B\|_{L_{\omega}^{q,\kappa}(X)}} < C.
    \end{equation}
\end{enumerate}
\end{theorem}

\begin{theorem}\label{thm main3}
Let $\kappa\in(0,1), \omega\in A_{p,q}(X)$ and $1< p < \frac{1}{\gamma} $ and $ \frac{1}{q} = \frac{1}{p} - \gamma$. Then for $b$ be a real valued, locally integrable function in $(X,d,\mu)$ the following conditions are equivalent:

\begin{enumerate}

\item[(i)] If $b\in BMO(X)$ and $b^-\in L^\infty(X),$
\item[(ii)] The commutator $[M_{\gamma}, b]$ is bounded from $L_{(\omega^p, \omega^q)}^{p,\kappa}(X)$ to $L_{\omega^q}^{q,\frac{\kappa q}{p}}(X)$.
\item[(iii)] There exists some constant $C$ such that 
    \begin{equation}\label{eqfrac}
        \sup_{B}\frac{\|(b - \mu(B)^{-\gamma}M_{\gamma,B}(b)\chi_B)\|_{L_{\omega^q}^{q,\frac{\kappa q}{p}}(X)}}{\|\chi_B\|_{L_{\omega^q}^{q,\frac{\kappa q}{p}}(X)}}< C.
    \end{equation}
\end{enumerate}
\end{theorem}

This paper is organized as follows. In Section \ref{s2}, we will give some preliminary definitions and recall some results on space of homogeneous type. In section 3 and 4 we describe the proofs of our boundedness results. Throughout this paper, we denote by $C$ and $\tilde{C}$, some positive constants which are independent of main parameters and they may vary from line to line.

\section{\bf Preliminaries on Spaces of Homogeneous Type}
\label{s2}
\noindent

We say that $(X,d,\mu)$ is a {space of homogeneous type} in the
sense of Coifman and Weiss (\cite{rg1}) if $d$ is a quasi-metric on~$X$
and $\mu$ is a nonzero measure satisfying the doubling
condition. A \emph{quasi-metric}~$d$ on a set~$X$ is a
function $d: X\times X\longrightarrow[0,\infty)$ satisfying
(i) $d(x,y) = d(y,x) \geq 0$ for all $x$, $y\in X$; (ii)
$d(x,y) = 0$ if and only if $x = y$; and (iii) the
\emph{quasi-triangle inequality}: there is a constant $A_0\in
[1,\infty)$ such that for all $x$, $y$, $z\in X$, 
\begin{eqnarray}\label{eqn:quasitriangleineq}
    d(x,y)
    \leq A_0 [d(x,z) + d(z,y)].
\end{eqnarray}
We say that a nonzero measure $\mu$ satisfies the
\emph{doubling condition} if there is a constant $C_\mu$ such
that for all $x\in X$ and $r > 0$,
\begin{eqnarray}\label{doubling condition}
   \mu(B(x,2r))
   \leq C_\mu \mu(B(x,r))
   < \infty,
\end{eqnarray}
where $B(x,r)$ is the quasi-metric ball given by $B(x,r) := \{y\in X: d(x,y)
< r\}$ for $x\in X$ and $r > 0$.We point out that the doubling condition (\ref{doubling
condition}) implies that there exists a positive constant
$n$ (the \emph{upper dimension} of~$\mu$)  such
that for all $x\in X$, $\lambda\geq 1$ and $r > 0$,
\begin{eqnarray}\label{upper dimension}
    \mu(B(x, \lambda r))
    \leq  C_\mu\lambda^{n} \mu(B(x,r)).
\end{eqnarray}


A subset $\Omega\subseteq X$ is \emph{open} (in the topology
induced by $\rho$) if for every $x\in\Omega$ there exists
$\eps>0$ such that $B(x,\eps)\subseteq\Omega$. A subset
$F\subseteq X$ is \emph{closed} if its complement $X\setminus
F$ is open. The usual proof of the fact that $F\subseteq X$ is
closed, if and only if it contains its limit points, carries
over to the quasi-metric spaces.


Throughout this paper we assume that $\mu(X) = \infty$ and that $\mu(x_0) = 0$ for every $x_0 \in X$.

\subsection{Muckenhoupt $A_p$ Weights}

  Let $\omega(x)$ be a nonnegative locally integrable function
  on~$X$. For $1 < p < \infty$, we
  say $\omega$ is an $A_p(X)$ \emph{weight}, written $\omega\in
  A_p$, if
  \begin{equation*}
    [\omega]_{A_p}
    := \sup_B \left(\frac{1}{\mu(B)}\int_B \omega(x)d\mu(x)\right)
    \left(\frac{1}{\mu(B)}\int_B \omega(x)^{1-p'}d\mu(x)\right)^{p-1}
    < \infty.
  \end{equation*}
  Here the suprema are taken over all balls~$B\subset X$.
  The quantity $[\omega]_{A_p}$ is called the \emph{$A_p$~constant
  of~$\omega$}.

\vspace{0.2 cm}
For $p = 1$, we say $\omega$ is an $A_1(X)$ \emph{weight},
  written $\omega\in A_1$, if $M(\omega)(x)\leq \omega(x)$ for $\mu$-almost every $x\in X$ and we define $A_{\infty}(X) = \cup_{1<p<\infty} A_p(X)$.

\vspace{0.2 cm}

We now define $A_{p,q}(X)$ weights (see, \cite{mw}). We say a weight function $\omega$ belongs to $A_{p,q}$, for $1<p<q<\infty,$ if there exists a constant $C>1$ such that

 \begin{equation*}
    [\omega]_{A_{p,q}}
    := \sup_B \left(\frac{1}{\mu(B)}\int_B \omega(x)^q d\mu(x)\right)^{1/q}
    \left(\frac{1}{\mu(B)}\int_B \omega(x)^{-p'}d\mu(x)\right)^{1/p'}
    < \infty
  \end{equation*}
  where $1/p+1/p'=1$.

  When $p=1$, $\omega\in A_{1,q}$ with $1<q<\infty$ if there exists $C>1$ such that 

  \begin{equation*}
    [\omega]_{A_{p,q}}
    := \sup_B \left(\frac{1}{\mu(B)}\int_B \omega(x)^q d\mu(x)\right)
    \left(ess sup_{x\in B}\frac{1}{\omega(x)}\right)^{1/p'}
    < \infty
  \end{equation*}


\vspace{0.2 cm}

We have the following characterization for the weights.

\begin{remark}\cite{ks,ldy}\label{weightorder}
    If $\omega \in A_{p,q}(X)$, with $1<p<q$, then the following are true:
    \begin{enumerate}
        \item $\omega^q \in A_t(X)$ with $t=1+q/p'$, where $1/p+1/p'=1.$
        \item $\omega^{-p'}\in A_{t'}$ with $1/t+1/t'=1.$
        \item $\omega\in A_{q,p}(X).$
        \item $\omega^p \in A_p(X)$ and $\omega^q\in A_q(X).$
    \end{enumerate}
\end{remark}

\subsection{BMO spaces}

Next we recall the definition of the weighted BMO space on space of homogeneous type, while we point out that the Euclidean version was first introduced by Muckenhoupt and Wheeden \cite{brr}. A function $b\in L^1_{\rm loc}(X)$ belongs to
the BMO space $BMO(X)$ if
\begin{equation*}
\|b\|_{BMO(X)}:=\sup_{B}{\frac{1}{\mu(B)}}\int_{B}
\left|b(x)-b_{B}\right|\, d\mu(x)<\infty,
\end{equation*}
where the sup is taken over all quasi-metric balls $B\subset X$ and
$$ b_B= {1\over\mu(B)} \int_B f(y)d\mu(y). $$

Note that we have the following equivalent norm for the BMO(X) functions,

\begin{equation*}
  \|b\|_{BMO(X)} \sim \frac{1}{\mu(B)}\inf_{c\in \mathbb{R}}\int_B|b(t)-c|d\mu(t)
\end{equation*}

\vspace{0.1 cm}


Below we list useful remarks and lemma's required in the proof of our main results. 

\begin{theorem}\cite{ks}[Theorem 3.1]\label{bhl}
    If $1<p<\infty, 0<\kappa<1$ and $\omega \in A_p(X),$ then the Hardy-Littlewood maximal function $M$ is bounded on $L_{\omega}^{p,\kappa}(X).$
\end{theorem}

We use this to prove the following lemma that gives the boundedness of $M_p$ on the weighted Morrey space $L_{\omega}^{q,\kappa}(X),$ for $1<p<q.$

\begin{lemma}\label{mqbound}
    If $1<p<\infty, 0<\kappa<1$ and $\omega \in A_p(X),$ then $M_p$ is bounded on $L_{\omega}^{q,\kappa}(X),$ for $p<q<\infty.$
\end{lemma}

\begin{proof}
     To prove this result, we will use the fact that $M_p(f)(t) = (M(|f|^p)(t))^{\frac{1}{p}}$. Now using the Theorem \ref{bhl}, we obtain

    \begin{align*}
      \|M_p(f)\|_{L_{\omega}^{q,\kappa}(X)}=\left(\frac{1}{\omega(B)^\kappa}\int_B |M_p(f)(t)|^q\omega(t)d\mu(t)\right)^{1/q} & = \left(\frac{1}{\omega(B)^\kappa}\int_B |M(|f|^p)(t)|^{\frac{q}{p}}\omega(t)d\mu(t)\right)^{1/q} \\&\leq \|M(|f|^p)\|_{L_{\omega}^{\frac{q}{p},\kappa}(X)}^{1/p}\\ &\lesssim \||f|^p\|_{L_{\omega}^{\frac{q}{p},\kappa}(X)}^{1/p} = \|f\|_{L_{\omega}^{q,\kappa}(X)}
    \end{align*}

    Hence we obtain the boundeness of $M_p$ is bounded on $L_{\omega}^{q,\kappa}(X),$.
\end{proof}


The following boundedness result easily follows from \cite{gvw}[Theorem 1.1], taking $\lambda_1= \lambda_2$.

\begin{lemma}\label{bmax}
    Suppose $1<p<\infty$, $\omega \in A_p(X)$, then for $b\in BMO(X)$, there exists a positive constant $C$ such that 
    \vspace{-0.2 cm}

    \begin{equation*}
        \|[b, M](f)\|_{L^p(X,\omega)} \lesssim C\|b\|_{BMO(X)} \|f\|_{{L^p(X,\omega)}}.
    \end{equation*}
\end{lemma}

\begin{remark}\label{bmp}
    Using $M_p(f)(t) = (M(|f|^p)(t))^{\frac{1}{p}}$ and the lemma above, we obtain similar boundedness result for the commutator $[b, M_p]$ when $b\in BMO(X)$ for any $p\in [1,\infty).$
\end{remark}

The following boundedness charaterization is known for the fractional maximal function on the weighted Morrey spaces.

\begin{theorem}\label{lemfrac2}\cite{ks}
    Let $\gamma,\kappa\in(0,1)$, $\omega\in A_{p,q}(X)$ and $1< p < \frac{1}{\gamma} $ then for $ \frac{1}{q} = \frac{1}{p} - \gamma$ we have

    \begin{equation*}
        \|M_{\gamma}f\|_{L_{\omega^q}^{q,\frac{\kappa q}{p}}(X)} \leq C \|f\|_{L_{(\omega^p, \omega^q)}^{p,\kappa}(X)},
    \end{equation*}

    for some constant $C$ independent of $f.$
\end{theorem}

Note that, the maximal commutator of fractional maximal function $M_\gamma$ with the locally integrable function $b$ is given by 

\vspace{-0.2 cm}

\begin{equation*}
    M_{\gamma,b} := \sup_{B \ni x} \frac{1}{\mu(B)^{1-\gamma}}\int_B |b(x)-b(y)||f(y)|\,d\mu(y),
\end{equation*}

where the supremum is taken over all balls $B\subset X$. Also observe that, the equation below holds (see, \cite{zws})

\begin{equation}\label{relmaxfrac}
    |[b, M_\gamma](f)(x)|\leq M_{b,\gamma}(f)(x) + 2 b^-(x)M_{\gamma}(f)(x)
\end{equation}

We have the following boundedness result for the maximal commutator of fractional maximal function.



\begin{lemma}\label{lemfrac1}\cite{vsg}
    Let $0<\gamma<1$ and $b\in BMO(X)$. Then there exists a constant $C$ such that
    \begin{equation*}
        M_{\gamma,b}f(x) \leq C\|b\|_{BMO(X)}(M(M_{\gamma}f)(x) + M_{\gamma}(Mf)(x))
    \end{equation*}
    holds for almost every $x\in X$ and any locally integrable function $f.$
\end{lemma}

Using the Theorem \ref{lemfrac2} and Lemma \ref{lemfrac1}, we can easily obtain the boundedness result described below, for $M_{\gamma,b}$.

\begin{lemma}\label{prop3}
Let $b \in BMO(X)$, $\kappa\in(0,1)$ and $\omega\in A_{p,q}(X)$,  then $M_{\gamma, b}$  is bounded from $L_{(\omega^p, \omega^q)}^{p,\kappa}(X)$ to $L_{\omega^q}^{q,\frac{\kappa q}{p}}(X)$ for $1< p < \frac{1}{\gamma} $ and $ \frac{1}{q} = \frac{1}{p} - \gamma$.
\end{lemma}

\begin{proof} 
Using the Lemma \ref{lemfrac1}, Lemma \ref{lemfrac2} and the Theorem \ref{bhl}, we obtain

\begin{align*}
    \|M_{\gamma,b}f\|_{L_{\omega^q}^{q,\frac{\kappa q}{p}}(X)} &\leq C\|b\|_{BMO(X)}\|(M(M_{\gamma}f)(x) + M_{\gamma}(Mf)(x))\|_{L_{\omega^q}^{q,\frac{\kappa q}{p}}(X)}\\
    &\lesssim  \|M_{\gamma}f\|_{L_{\omega^q}^{q,\frac{\kappa q}{p}}(X)} + \|M(f)\|_{L_{(\omega^p, \omega^q)}^{p,\kappa}(X)}\\
    &\lesssim \|f\|_{L_{(\omega^p, \omega^q)}^{p,\kappa}(X)}.
\end{align*}

Hence, we obtain the boundedness of $M_{\gamma,b}$ from $L_{(\omega^p, \omega^q)}^{p,\kappa}(X)$ to $L_{\omega^q}^{q,\frac{\kappa q}{p}}(X)$.
    
\end{proof}


\section{\bf Proof of the Theorems \ref{thm main1} and Theorem \ref{thmax}}

To give the proof of the Theorem \ref{thm main1}, we will use the following proposition.

\begin{proposition}\label{prop1}
Let $b \in BMO(X)$ be a nonnegative function. Then for $\kappa\in(0,1), \omega\in A_{p}(X)$,  $[M_p, b]$  is bounded on $L_{\omega}^{q,\kappa}(X)$, $1< p<  q < \infty$.
\end{proposition}

\begin{proof}
    Using the Lemma \ref{bmax} and Remark \ref{bmp}, we have that  $[M_p, b]$ is bounded on the weighted $L^q(X,\omega)$ space, i.e.
    \begin{equation*}
        \|[M_p, b](f)\|_{L^q(X,\omega)} \leq \|f\|_{L^q(X,\omega)}.
    \end{equation*}

Now let us fix a ball $B$ and consider a function $h\in L_{\omega}^{q,\kappa}(X)$, and let $f = \frac{h\chi_B}{\omega(B)^{\kappa/q}}$. We can observe that $f\in L^q(X)$ and

\begin{equation*}
    \|f\|_{L^q(X,\omega)} = \left(\frac{1}{\omega(B)^\kappa}\int_{B}|h(t)|^q \omega(t)d\mu(t)\right)^{1/q}.
\end{equation*}

Now observe that

\begin{equation}\label{eq1}
    \left(\int_{B}|b(t)M_p(f)(t)-M_p(bf)(t)|^q \omega(t)d\mu(t)\right)^{1/q} \leq \|[M_p, b](f)\|_{L^q(X,\omega)}
\end{equation}

Since for $t\in B$, we have

\begin{equation}\label{eq2}
    M_p(bf)(t) = \frac{1}{\omega(B)^{\kappa/q}}M_p(bh)(t) \quad \textit{and} \quad M_p(f)(t) = \frac{1}{\omega(B)^{\kappa/q}}M_p(h)(t)
\end{equation}

Hence using \ref{eq2} and plugging this in \ref{eq1} we get,

\begin{align*}
 \left(\frac{1}{\omega(B)^{\kappa}} \int_{B}|b(t)M_p(h)(t)-M_p(bh)(t)|^q \omega(t)d\mu(t)\right)^{1/q} &\leq \|[M_p, b](f)\|_{L^q(X,\omega)}\\
 & \lesssim \|f\|_{L^q(X,\omega)}\\ & \leq \left(\frac{1}{\omega(B)^\kappa}\int_{B}|h(t)|^q \omega(t)d\mu(t)\right)^{1/q}
\end{align*}

Therefore

\begin{equation*}
        \|[M_p, b](h)\|_{L_{\omega}^{q,\kappa}(X)} \lesssim \|h\|_{L_{\omega}^{q,\kappa}(X)}.
    \end{equation*}
\end{proof}

\subsection{Proof of Theorem \ref{thm main1}} 
\begin{proof}
    To obtain the equivalence, we will first show $(i)\implies (ii)$. Let us assume that $b\in BMO(X)$ and $b^-\in L^\infty(X)$. If  $b\in BMO(X)$ then we know that $|b|$ is also in $BMO(X)$. Now observe that, we have

\begin{equation*}
    |[M_p,b](f) - [M_p,|b|](f)|\leq 2b^-M_p(f).
\end{equation*}

Now using the Proposition \ref{prop1}, the Lemma \ref{mqbound} and  $b^-\in L^\infty(X)$, we obtain $(i)$.

\vspace{0.2 cm}

To prove the implication $(ii) \implies (iii)$, for a fixed ball $B$, we let $f= \chi_B \in L_{\omega}^{q,\kappa}(X)$, then $\|f\|_{L_{\omega}^{q,\kappa}(X)} = \omega(B)^{\frac{1-\kappa}{q}}$.

\vspace{0.1 cm}
By $(ii)$, we have

$$\|[M_p, b](f)\|_{L_{\omega}^{q,\kappa}(X)} \leq C\|f\|_{L_{\omega}^{q,\kappa}(X)} = C \omega(B)^{\frac{1-\kappa}{q}}$$

\vspace{0.1 cm}
As we know that $M_p(b\chi_B) = M_{p, B}(b)$ and $M_p(\chi_B) = \chi_B$, we obtain
\begin{align*}
    \left(\frac{1}{\omega(B)^{\kappa}} \int_{B}|b(t)-M_{p,B}(b)(t)|^q \omega(t)d\mu(t)\right)^{1/q} &\leq \left(\frac{1}{\omega(B)^{\kappa}} \int_{B}|b(t)M_p(\chi_B)(t)-M_{p}(b\chi_B)(t)|^q \omega(t)d\mu(t)\right)^{1/q} \\
    &\leq \|[M_p, b](f)\|_{L_{\omega}^{q,\kappa}(X)} \leq \|\chi_B\|_{L_{\omega}^{q,\kappa}(X)}. 
\end{align*}

So, we have

\begin{equation*}
    \left(\frac{1}{\omega(B)^{\kappa}} \int_{B}|b(t)-M_{p,B}(b)(t)|^q \omega(t)d\mu(t)\right)^{1/q} \leq C \omega(B)^{\frac{1-\kappa}{q}}.
\end{equation*}

That gives us

\begin{align*}
     \left(\frac{1}{\omega(B)} \int_{B}|b(t)-M_{p,B}(b)(t)|^q \omega(t)d\mu(t)\right)^{1/q} &\leq C 
\end{align*}

Hence we obtain,

\begin{equation*}
\sup_{B}  \frac{\|(b - M_{p,B}(b))\chi_B\|_{L_{\omega}^{q,\kappa}(X)}}{\|\chi_B\|_{L_{\omega}^{q,\kappa}(X)}}  =  \sup_B\left(\frac{1}{\omega(B)} \int_{B}|b(t)-M_{p,B}(b)(t)|^q \omega(t)d\mu(t)\right)^{1/q} < C.
\end{equation*}
\vspace{0.1 cm}

Now to prove the inclusion $(iii) \implies (i),$ we will need show that $b\in BMO(X).$ 

\vspace{0.1 cm}

To proceed, fix a cube $B$ and let $E= \{t\in B: b(t) \leq b_B\}$ and $F= \{t\in B: b(t) > b_B\}$, then we have

$$\int_E |b(t) - b_{B}|d\mu(t) = \int_F |b(t) - b_{B}|d\mu(t)$$.

Observe for all $t\in E$, we have $b(t)\leq b_B \leq M_{p, B}(b)(t)$.  Hence, we obtain that

\begin{align}\label{eqbound}
      \frac{1}{\mu(B)}\int_B|b(t) - b_{B}|d\mu(t) & =   \frac{2}{\mu(B)}\int_E|b(t) - b_{B}|d\mu(t)\nonumber\\
      &\leq  \frac{2}{\mu(B)}\int_E|b(t) - M_{p,B}|d\mu(t)\nonumber\\
      &\leq  \frac{2}{\mu(B)}\int_B|b(t) - M_{p,B}|d\mu(t)\nonumber\\
      &=\frac{2}{\mu(B)}\int_B|b(t) - M_{p, B}(b)(t)|\omega(t)^{\frac{1}{q}}\omega(t)^{\frac{-1}{q}}d\mu(t)\quad (\textit{using H\"older's inequality})\nonumber\\
    &\leq 2\left(\frac{1}{\mu(B)}\int_B|b(t) - M_{p, B}(b)(t)|^q \omega(t)d\mu(t)\right)^{\frac{1}{q}} \times \left(\frac{1}{\mu(B)}\int_B \omega(t)^{\frac{-q'}{q}}d\mu(t)\right)^{\frac{1}{q'}}\nonumber\\
    &\lesssim 2[\omega]_{A_q}^{1/q}\left(\frac{1}{\mu(B)}\int_B|b(t) - M_{p, B}(b)(t)|^q \omega(t)d\mu(t)\right)^{\frac{1}{q}} \times \left(\frac{1}{\mu(B)}\int_B \omega(t)d\mu(t)\right)^{\frac{-1}{q}}\nonumber\\
    &\lesssim 2 [\omega]_{A_q}^{1/q}\left(\frac{1}{\omega(B)}\int_B|b(t) - M_{p, B}(b)(t)|^q \omega(t)d\mu(t)\right)^{\frac{1}{q}}\nonumber \\&\lesssim 2 [\omega]_{A_q}^{1/q}\sup_{B}  \frac{\|(b - M_{p,B}(b))\chi_B\|_{L_{\omega}^{q,\kappa}(X)}}{\|\chi_B\|_{L_{\omega}^{q,\kappa}(X)}}< \infty.
\end{align}

In the above computations we used that if $\omega \in A_p$ then $\omega \in A_q$ for $p<q$. Hence we obtain $b\in BMO(X)$.

To show $b^-\in L^\infty(X)$, note that we have 

$$ 0\leq b^-  = |b|-b^+ \leq M_{p,B}(b)-b^+ + b^- =  M_{p,B}(b)-b.$$ 

Here we used that $|b|\leq M_{p,B}(b)$ in B. 

Now using the $(iii)$ and using the computation given in \eqref{eqbound}, we have that $(b^{-})_B \leq C.$ Hence the boundedness of $b^-$ follows from Lebesgue differentiation theorem.
\end{proof}

Now let us prove the boundedness result for the maximal commutator of Hardy-Littlewood maximal function.

\subsection{Proof of Theorem \ref{thmax}}

To prove the sufficiency, that is if $b\in BMO(X)$, to show the boundedness of $C_b$ on $L_{\omega}^{q,\kappa}(X)$ we will use the Lemma \ref{bhl} and the following inequality (given in \cite{mgkm})

\begin{equation}
    C_b(f)(x) \leq C\|b\|_{BMO(X)}M^2(f)(x)
\end{equation}

Therefore we obtain,

\begin{equation*}
    \|C_b(f)\|_{L_{\omega}^{q,\kappa}(X)}\leq C \|b\|_{BMO(X)} \|M^2(f)\|_{L_{\omega}^{q,\kappa}(X)} \leq C \|b\|_{BMO(X)}\|f\|_{L_{\omega}^{q,\kappa}(X)}.
\end{equation*}
\vspace{0.1 cm}

To prove the necessity,  let us assume $C_b$ is bounded on $L_{\omega}^{q,\kappa}(X)$, we want to show that $b\in BMO(X).$

Observe

\begin{align*}
\frac{1}{\mu(B)}\inf_{c\in \mathbb{R}}\int_B|b(t)-c|d\mu(t) &\leq \frac{1}{\mu(B)}\inf_{x\in B}\int_B|b(t)-b(x)|d\mu(t) \\
&\leq \frac{1}{\mu(B)^2}\int_B\int_B|b(t)-b(x)|d\mu(t)d\mu(x)\\
&\leq \frac{1}{\mu(B)^2}\int_B\int_B|b(t)-b(x)|\omega(t)^{\frac{1}{p}}\omega(t)^{\frac{-1}{p}}d\mu(t)d\mu(x)\\
&\leq  \left(\frac{1}{\mu(B)}\int_B\left(\frac{1}{\mu(B)}\int_B|b(t)-b(x)|d\mu(x)\right)^p\omega(t)d\mu(t)\right)^{\frac{1}{p}}\left(\frac{1}{\mu(B)}\int_B\omega(t)^{\frac{-p'}{p}}d\mu(t)\right)^{1/p'}\\
&\lesssim [\omega]_{A_p}^{1/p}\left(\frac{1}{\mu(B)}\int_B\omega(t)d\mu(t)\right)^{\frac{-1}{p}}\left(\frac{1}{\mu(B)}\int_B\left(\frac{1}{\mu(B)}\int_B|b(t)-b(x)|d\mu(x)\right)^p\omega(t)d\mu(t)\right)^{\frac{1}{p}}\\
&\lesssim [\omega]_{A_p}^{1/p}(\omega(B))^{\frac{-1}{p}}\left(\int_B\left(\frac{1}{\mu(B)}\int_B|b(t)-b(x)|d\mu(x)\right)^p\omega(t)d\mu(t)\right)^{\frac{1}{p}}\\
&\lesssim [\omega]_{A_p}^{1/p}(\omega(B))^{\frac{-1}{p}}\left(\int_B|C_b(\chi_B)|^p\omega(t)d\mu(t)\right)^{\frac{1}{p}}\\
&\lesssim [\omega]_{A_p}^{1/p} (\omega(B))^{\frac{\kappa-1}{p}} \|C_b(\chi_B)\|_{L_{\omega}^{q,\kappa}(X)}\\
&\lesssim [\omega]_{A_p}^{1/p} (\omega(B))^{\frac{\kappa-1}{p}} \|\chi_B\|_{L_{\omega}^{q,\kappa}(X)}\\
&\lesssim [\omega]_{A_p}^{1/p} (\omega(B))^{\frac{\kappa-1}{p}}(\omega(B))^{\frac{1-\kappa}{q}} \leq C. 
\end{align*}

Here we used that $\frac{1}{\mu(B)}\int_B|b(t)-b(x)|d\mu(x) \leq C_b(\chi_B)(t)$ in the inequalities above. Hence we obtain $b\in BMO(X).$

\section{Proof of Theorem \ref{thm main2}}

We need the following proposition to prove Theorem \ref{thm main2}.

\begin{proposition}\label{prop2}
  Let $b \in BMO(X)$ be a nonnegative function. Then for $\kappa\in(0,1), \omega\in A_{1}(X)$,  $[ M^{\sharp}, b]$  is bounded on $L_{\omega}^{q,\kappa}(X)$, $1< q < \infty.$  
\end{proposition}

\begin{proof}
  Consider
  \begin{align*}
      |[ M^{\sharp}, b](f)(t)| &= |b(t)M^{\sharp}(f)(t)-M^{\sharp}(bf)(t)|\\
      & =  \left| b(t)\sup_{B \ni t} {1\over \mu(B)}\int_B |f(y) - (f)_B|d\mu(y)-\sup_{B \ni t} {1\over \mu(B)}\int_B |b(y)f(y) - (bf)_B|d\mu(y)\right|\\
&\leq \sup_{B \ni t} {1\over \mu(B)}\int_B |b(t)f(y)-b(t)(f)_B-b(y)f(y) +(bf)_B|d\mu(y)\\
&\leq \sup_{B \ni t} {1\over \mu(B)}\int_B |b(y)f(y) - b(t)f(y)|d\mu(y)+\sup_{B \ni t} {1\over \mu(B)}\int_B|
b(t)(f)_B - (bf)_B| d\mu(y)\\
&\leq 2 \sup_{B \ni t} {1\over \mu(B)}\int_B |b(y) - b(t)|f(y)d\mu(y)\\
&\leq 2 \sup_{B \ni t}\left( {1\over \mu(B)}\int_B |b(y) - b(t)|^{r'}d\mu(y)\right)^{1/r'}\left({1\over \mu(B)}|f(y)|^r d\mu(y)\right)^{1/r}\\
&\leq \|b\|_{BMO(X)}M_r(f)(t).
\end{align*}

We used H\"older's inequality in the second to last step above, where $p<r<q$ and $\frac{1}{r}+\frac{1}{r'}=1$. Therefore, we obtain
$$\|[ M^{\sharp}, b](f)\|_{ L_{\omega}^{q,\kappa}(X)}  \lesssim  \|b\|_{BMO(X)}\|M_r(f)\|_{L_{\omega}^{q,\kappa}(X)} \lesssim  \|b\|_{BMO(X)}\|f\|_{L_{\omega}^{q,\kappa}(X)}.$$

Here we used that $M_r(f)$ is bounded on $L_{\omega}^{q,\kappa}(X)$ which follows from the Lemma \ref{mqbound}.

\end{proof}

\subsection{Proof of Theorem \ref{thm main2}}

\begin{proof}
    We will prove the equivalence by showing $(i)\implies (ii) \implies (iii)\implies (i).$

The inclusion $(i)\implies (ii)$,  follows in the same lines as the proof of the Theorem \ref{thm main1}. For any quasi-linear operator $T$ and general BMO function $b$ we have the following

\begin{align}\label{eqine}
    |[T,b](f)-[T,|b|](f)|\leq 2(b^- T(f) + T(b^- f))
\end{align}

As $M^{\sharp}$ is a quasi-linear operator, \eqref{eqine} holds true. Now using $M^{\sharp}(f) \leq 2 M(f)$ for all locally integrable functions $f$, and the Proposition \ref{prop2} and that $b^-\in L^\infty(X)$, $(i)$ follows immediately.

Now we prove the implication $(ii)\implies (iii)$. Let us fix a ball $Q$ and assume $\chi_Q \in L_{\omega}^{q,\kappa}(X)$. We know that $M^{\sharp}(\chi_Q) = \frac{1}{2}$, now using $(ii)$, we have 

\begin{align*}
\|[M^{\sharp}, b](\chi_Q)\|_{L_{\omega}^{q,\kappa}(X)} &\leq C\|\chi_Q\|_{L_{\omega}^{q,\kappa}(X)}
\end{align*}

which is same as

\begin{align*}
 \left(\frac{1}{\omega(Q)^{\kappa}} \int_{Q}|b(t)M^{\sharp}(\chi_Q)(t)-M^{\sharp}(b\chi_Q)(t)|^q \omega(t)d\mu(t)\right)^{1/q}   &\leq C\omega(Q)^{\frac{1-\kappa}{q}}\\
 \left(\frac{1}{\omega(Q)} \int_{Q}|b(t)-2M^{\sharp}(b\chi_Q)|^q \omega(t)d\mu(t)\right)^{1/q} \leq \tilde{C}
\end{align*}

Hence we obtain $(iii)$

\begin{align*}
    \sup_{Q}\frac{\|(b - 2M^{\sharp}(b))\chi_Q\|_{L_{\omega}^{q,\kappa}(X)}}{\|\chi_Q\|_{L_{\omega}^{q,\kappa}(X)}}=\sup_Q \left(\frac{1}{\omega(Q)} \int_{Q}|b(t)-2M^{\sharp}(b\chi_Q)|^q \omega(t)d\mu(t)\right)^{1/q}  < \infty.
\end{align*}

Now we finally prove $(iii) \implies (i)$. The proof here is similar to proof of Theorem \ref{thm main1}. Consider for a ball $Q$, a ball $\tilde{Q}$ containing $Q$ such that $\mu(\tilde{Q}) = 2\mu(Q)$. Suppose that $x\in Q$, then we have 

\begin{align*}
    \frac{1}{2\mu(Q)} \int_{Q}|b(t)-\frac{1}{2}b_Q|d\mu(t) + \frac{1}{4}|b_Q| & = \frac{1}{2\mu(Q)} \left(  \int_{Q}|b(t)-\frac{1}{2}b_Q|d\mu(t) + \frac{1}{2} \mu(\tilde{Q} \setminus Q)|b_Q|\right)\\
    &=\frac{1}{\mu(\tilde{Q})}\int_{Q}|b\chi_Q(t)-(b\chi_Q)_{\tilde{Q}}|d\mu(t)+ \int_{\tilde{Q}\setminus Q}|b\chi_Q(t)-(b\chi_Q)_{\tilde{Q}}|d\mu(t) \\
    &=\frac{1}{\mu(\tilde{Q})}\int_{\tilde{Q}}|b\chi_Q(t)-(b\chi_Q)_{\tilde{Q}}|d\mu(t)\\
    & \leq M^{\sharp}(b\chi_Q)(x).
\end{align*}

Here we used that $(b\chi_Q)_{\tilde{Q}} = \frac{1}{2}b_{Q}$. Now, observe that

\begin{align*}
    |b_Q| &\leq \frac{1}{\mu(Q)} \int_{Q}|b(t)-\frac{1}{2}b_Q|d\mu(t) + \frac{1}{\mu(Q)}\int_Q |\frac{1}{2}b_Q|d\mu(t)\\
    & = \frac{1}{\mu(Q)}\left( \int_{Q}|b(t)-\frac{1}{2}b_Q|d\mu(t) + \frac{1}{2}|b_Q|\right),
\end{align*}

So, we get
\begin{align*}
    |b_Q|\leq \frac{1}{\mu(Q)} \int_{Q}|b(t)-\frac{1}{2}b_Q|d\mu(t)+\frac{1}{2}|b_Q| \leq  2M^{\sharp}(b\chi_Q)(x) \quad \forall x\in Q.
\end{align*}

Now will now show $b\in BMO(X)$. Let us consider $E=\{x\in Q: b(x)\leq b_Q\},$ then using the same strategy as used in the proof of the Theorem \ref{thm main1}, we have

\begin{align*}
    \frac{1}{\mu(Q)}\int_Q|b(t) - b_{Q}|d\mu(t) & =  \frac{2}{\mu(Q)}\int_E (b_{Q}-b(t))d\mu(t) \\
    &\leq  \frac{2}{\mu(Q)}\int_E|2M^{\sharp}(b\chi_Q)(t)-b(t)|d\mu(t)  \\
    & \leq \frac{2}{\mu(Q)}\int_E|2M^{\sharp}(b\chi_Q)(t)-b(t)|\omega(t)^{\frac{1}{q}}\omega(t)^{\frac{-1}{q}}d\mu(t)\\
    & \leq 2\left(\frac{1}{\mu(Q)}\int_B|b(t) - 2M^{\sharp}(b\chi_Q)(t)|^q \omega(t)|d\mu(t)\right)^{\frac{1}{q}} \times \left(\frac{1}{\mu(Q)}\int_B \omega(t)^{\frac{-q'}{q}}d\mu(t)\right)^{\frac{1}{q'}}\\
    &\lesssim 2 [\omega]_{A_q}^{1/q}\left(\frac{1}{\mu(Q)}\int_B|b(t) - 2M^{\sharp}(b\chi_Q)(t)|^q \omega(t)|d\mu(t)\right)^{\frac{1}{q}} \times \left(\frac{1}{\mu(Q)}\int_B \omega(t)d\mu(t)\right)^{\frac{-1}{q}}\\
    &\lesssim 2 [\omega]_{A_q}^{1/q}\left(\frac{1}{\omega(Q)}\int_B|b(t) - 2M^{\sharp}(b\chi_Q)(t)|^q \omega(t)|d\mu(t)\right)^{\frac{1}{q}}\\&\leq 2 [\omega]_{A_q}^{1/q} \sup_{Q}\frac{\|(b - 2M^{\sharp}(b))\chi_Q\|_{L_{\omega}^{q,\kappa}(X)}}{\|\chi_Q\|_{L_{\omega}^{q,\kappa}(X)}}
    < \infty.
\end{align*}

\vspace{2 cm}

To show that b is bounded, consider

\begin{align*}
    |b_Q| - b^+(t) + b^-(t) \leq 2M^{\sharp}(b\chi_Q)(t) - b(t)
\end{align*}

Now averaging over $Q$, we have

\begin{align*}
    |b_Q| - \frac{1}{\mu(Q)}\int_Q b^+(t) d\mu(t)  + \frac{1}{\mu(Q)}\int_Q b^-(t)d\mu(t) &= \frac{1}{\mu(Q)}\int_Q ( |b_Q| - b^+(t) + b^-(t))d\mu(t) \\
    &\leq \frac{1}{\mu(Q)}\int_Q |2M^{\sharp}(b\chi_Q)(t) - b(t)|d\mu(t)\\
    &\leq \frac{1}{\mu(Q)}\int_Q |2M^{\sharp}(b\chi_Q)(t) - b(t)|\omega(t)^{\frac{1}{q}}\omega(t)^{\frac{-1}{q}}d\mu(t)\\
    & \lesssim \left(\frac{1}{\omega(Q)}\int_B|b(t) - 2M^{\sharp}(b\chi_Q)(t)|^q \omega(t)|d\mu(t)\right)^{\frac{1}{q}}\\
    &\leq C
\end{align*}

Let $\mu(Q) \rightarrow 0$ with $t\in Q,$ then using the Lebesgue differentiation theorem we obtain the following

\begin{equation*}
    2b^-(t) = |b(t)|- b^+(t) + b^-(t) \leq C.
\end{equation*}

\end{proof}



\subsection{Proof of Theorem \ref{thm main3}.} To show the equivalence, we will first show $(i) \implies (ii).$ Assume $b\in BMO(X)$ and $b^- \in L^{\infty}(X).$ Now using the equation \ref{relmaxfrac} and Lemma \ref{lemfrac2}, Lemma \ref{prop3}, we obtain the following for any $f\in L_{loc}^1(X)$

\begin{align*}
   \| [b, M_{\gamma}]f\|_{L_{\omega^q}^{q,\frac{\kappa q}{p}}(X)} &\leq \|M_{b,\gamma}(f)(x) + 2b^-(x)M_{\gamma}f(x)\|_{L_{\omega^q}^{q,\frac{\kappa q}{p}}(X)}\\
   &\leq \|M_{b,\gamma}(f)\|_{L_{\omega^q}^{q,\frac{\kappa q}{p}}(X)} + \|2b^-(x)M_{\gamma}f(x)\|_{L_{\omega^q}^{q,\frac{\kappa q}{p}}(X)}\\
   &\leq \|b\|_{BMO(X)} \|f\|_{L_{(\omega^p,\omega^q)}^{p,\kappa}(X)} + \|b^-\|_{L^{\infty}(X)} \|f\|_{L_{(\omega^p,\omega^q)}^{p,\kappa}(X)}\\
   &\lesssim \|f\|_{L_{(\omega^p, \omega^q)}^{p,\kappa}(X)}.
   \end{align*}

This gives us $(ii).$

Now let us show $(ii) \implies (iii)$. Assume there exists some $p$ and $q$ such that $1< p < \frac{1}{\gamma} $ and $ \frac{1}{q} = \frac{1}{p} - \gamma$ and that $[b, M_{\gamma}]$ is bounded from $L_{(\omega^p,\omega^q)}^{p,\kappa}(X)$ to $L_{\omega^q}^{q,\frac{\kappa q}{p}}(X)$. We will now show there is a constant $C$ such that \eqref{eqfrac} holds. Note that

\begin{align*}
 \|(b - \mu(Q)^{-\gamma}M_{\gamma,Q}(b)\chi_Q)\|_{L_{\omega^q}^{q,\frac{\kappa q}{p}}(X)} &= \left(\frac{1}{(\omega^q(Q))^\frac{\kappa q}{p}} \int_Q|b(x)-\mu(Q)^{-\gamma}M_{\gamma,Q}(b)(x)|^q \omega^q(x) d\mu(x)\right)^{\frac{1}{q}}\\ 
   & = \frac{1}{\mu(Q)^{\gamma}}\left( \frac{1}{(\omega^q(Q))^\frac{\kappa q}{p}}\int_Q|b(x)\mu(Q)^{\gamma}-M_{\gamma,Q}(b)(x)|^q \omega^q(x) d\mu(x)\right)^{\frac{1}{q}}
\end{align*}

\vspace{0.2 cm}

For any given ball $Q$ and given any $x\in Q$, we can easily show

\begin{equation*}
    M_{\gamma}(b\chi_Q)(x) = M_{\gamma, Q}(b)(x) \quad \textit{and}  \quad M_{\gamma}(\chi_Q)(x) = \mu(Q)^{\gamma}.
\end{equation*}

So, we have 

\begin{align}\label{eqcommu}
  |b(x)\mu(Q)^{\gamma}-M_{\gamma,Q}(b)(x)|^q &= 
  |b(x)M_{\gamma}(\chi_Q)(x)- M_{\gamma}(b\chi_Q)(x)|^q = |[b, M_{\gamma}](\chi_Q)(x)|^q .
\end{align}

Hence, using $(ii)$ and the equation \eqref{eqcommu}, it follows that 

\begin{align*}
     \frac{\|(b(x) - \mu(Q)^{-\gamma}M_{\gamma,Q}(b)\chi_Q)\|_{L_{\omega^q}^{q,\frac{\kappa q}{p}}(X)}}{\|\chi_Q\|_{L_{\omega^q}^{q,\frac{\kappa q}{p}}(X)}} &= \frac{\frac{1}{\mu(Q)^{\gamma}}\left( \frac{1}{(\omega^q(Q))^\frac{\kappa q}{p}}\int_Q|b(x)\mu(Q)^{\gamma}-M_{\gamma,Q}(b)(x)|^q \omega^q(x) d\mu(x)\right)^{\frac{1}{q}}}{\|\chi_Q\|_{L_{\omega^q}^{q,\frac{\kappa q}{p}}(X)}}\\
     &=\frac{\frac{1}{\mu(Q)^{\gamma}}\left(\frac{1}{(\omega^q(Q))^\frac{\kappa q}{p}} \int_Q|[b, M_{\gamma}](\chi_Q)(x)|^q \omega^q(x) d\mu(x)\right)^{\frac{1}{q}}}{\|\chi_Q\|_{L_{\omega^q}^{q,\frac{\kappa q}{p}}(X)}}\\
    &= \frac{\frac{1}{\mu(Q)^{\gamma}}\|[b, M_{\gamma}](\chi_Q)\|_{L_{\omega^q}^{q,\frac{\kappa q}{p}}(X)}}{\|\chi_Q\|_{L_{\omega^q}^{q,\frac{\kappa q}{p}}(X)}}
     \lesssim \frac{\frac{1}{\mu(Q)^{\gamma}}\|\chi_Q\|_{L_{(\omega^p,\omega^q)}^{p,\kappa}(X)}}{\|\chi_Q\|_{L_{\omega^q}^{q,\frac{\kappa q}{p}}(X)}}\\
     &= \frac{\frac{1}{\mu(Q)^{\gamma}}\left(\frac{1}{(\omega^q(Q))^\kappa}\int_Q\omega^p(x) d\mu(x)\right)^{\frac{1}{p}}}{\left(\frac{1}{(\omega^q(Q))^\frac{\kappa q}{p}}\int_Q\omega^q(x) d\mu(x)\right)^{\frac{1}{q}}}\\
     &= \frac{\frac{1}{\mu(Q)^{\gamma}}\left(\int_Q\omega^p(x) d\mu(x)\right)^{\frac{1}{p}}}{\left(\int_Q\omega^q(x) d\mu(x)\right)^{\frac{1}{q}}} = \frac{\frac{1}{\mu(Q)^{\frac{1}{p}-\frac{1}{q}}}\left(\int_Q\omega^p(x) d\mu(x)\right)^{\frac{1}{p}}}{\left(\int_Q\omega^q(x) d\mu(x)\right)^{\frac{1}{q}}} \\
     &= \frac{\left(\frac{1}{\mu(Q)}\int_Q\omega^p(x) d\mu(x)\right)^{\frac{1}{p}}}{\left(\frac{1} {\mu(Q)}\int_Q\omega^q(x) d\mu(x)\right)^{\frac{1}{q}}}\\&\lesssim \left(\frac{1}{\mu(Q)}\int_Q\omega^p(x) d\mu(x)\right)^{\frac{1}{p}} \left(\frac{1}{\mu(Q)}\int_Q\omega^{-q'}(x) d\mu(x)\right)^{\frac{1}{q'}} [w^q]_{A_q}^{\frac{1}{q}} \\
     &= [w]_{A_{q,p}} [w^q]_{A_q}^{\frac{1}{q}} < \infty.
\end{align*}

In the above computation we used that $\gamma = \frac{1}{p}-\frac{1}{q}$ and the fact that if $\omega \in A_{p,q}$ then we have that $\omega \in A_{q,p}(X)$ and $\omega^q \in A_q(X)$ given by the Remark \ref{weightorder}. Hence we obtain $(iii)$.

\vspace{0.1 cm}

Now we will show $(iii) \implies (i)$, i.e. we assume \eqref{eqfrac} holds, then show that implies $b \in BMO(X)$ and $b^- \in L^{\infty}(X).$

\vspace{0.1 cm}

To show $b\in BMO(X)$, first fix any ball $Q\in X$ and let $E=\{x\in Q: b(x)\leq b_Q\}$. Then we have

\begin{align*}
    \|b\|_{BMO(X)}=\frac{1}{\mu(Q)}\int_Q|b(x)-b_Q|d\mu(x) =  \frac{2}{\mu(Q)}\int_E (b_{Q}-b(x))d\mu(x) \\\leq \frac{2}{\mu(Q)}\int_E  |b(x)-\mu(B)^{-\gamma}M_{\gamma, B}(x)|d\mu(x).
\end{align*}

The last inequality here follows as we have $|b_Q|\leq \mu(Q)^{-\gamma}M_{\gamma, Q}(x),$ for all $x\in Q$ (see, \cite[Lemma 2.8]{wz}).

\vspace{0.2 cm}

Hence we have



\begin{align*}
      \|b\|_{BMO(X)}&=\frac{1}{\mu(Q)}\int_Q|b(x)-b_Q|d\mu(x)\\
      & \leq \frac{2}{\mu(Q)}\int_E|b(x)-\mu^{-\gamma}(Q)M_{\gamma,Q}(b)(x)|d\mu(t)\\
      & \leq \frac{2}{\mu(Q)}\int_Q|b(x)-\mu^{-\gamma}(Q)M_{\gamma,Q}(b)(x)|d\mu(t)\\
    &= \frac{2}{\mu(Q)}\int_Q|b(x)-\mu^{-\gamma}(Q)M_{\gamma,Q}(b)(x)|\omega(x)\omega^{-1}(x)d\mu(x)\\
& \leq 2 \left(\frac{\omega^q(Q))^\frac{\kappa q}{p}}{\mu(Q)}\frac{1}{(\omega^q(Q))^\frac{\kappa q}{p}} \int_Q|b(x)-\mu^{-\gamma}(Q)M_{\gamma,Q}(b)(x)|^q\omega^q(x)d\mu(x)\right)^{\frac{1}{q}} \left(\frac{1}{\mu(Q)}\int_Q \omega^{-q'}(x)d\mu(x)\right)^{\frac{1}{q'}}\\
   &= 2\left(\frac{\omega^q(Q))^\frac{\kappa }{p}}{\mu(Q)^{\frac{1}{q}}}\right)\|(b(x) - \mu(Q)^{-\gamma}M_{\gamma,Q}(b)\chi_Q)\|_{L_{\omega^q}^{q,\frac{\kappa q}{p}}(X)}\left(\frac{1}{\mu(Q)}\int_Q \omega^{-q'}(x)d\mu(x)\right)^{\frac{1}{q'}}\\
   & \lesssim 2\left(\frac{\omega^q(Q))^\frac{\kappa }{p}}{\mu(Q)^{\frac{1}{q}}}\right)\|\chi_Q\|_{L_{\omega^q}^{q,\frac{\kappa q}{p}}(X)}\left(\frac{1}{\mu(Q)}\int_Q \omega^{-q'}(x)d\mu(x)\right)^{\frac{1}{q'}}\\& = 2\left( \frac{\omega^q(Q))^\frac{\kappa }{p}}{\mu(Q)^{\frac{1}{q}}}\right)\left(\frac{1}{(\omega^q(Q))^\frac{\kappa q}{p}}\int_Q\omega^q(x) d\mu(x)\right)^{\frac{1}{q}}\left(\frac{1}{\mu(Q)}\int_Q \omega^{-q'}(x)d\mu(x)\right)^{\frac{1}{q'}}\\
   &= 2\left(\frac{1}{\mu(Q)}\int_Q\omega^q(x) d\mu(x)\right)^{\frac{1}{q}}\left(\frac{1}{\mu(Q)}\int_Q \omega^{-q'}(x)d\mu(x)\right)^{\frac{1}{q'}} = 2[w^q]_{A_q}^{\frac{1}{q}} < \infty.
\end{align*}

The last step in the above inequality follows from the Remark \ref{weightorder}. This concludes that $b\in BMO(X)$. The proof to show $b^-\in L^\infty(X)$ follows exactly as shown in the proof of Theorem \ref{thm main2}.



\setcounter{equation}{0}
\label{s:classical-to-Meyer-spaces}

\vspace{0.2 cm}

\noindent Manasa N. Vempati, Department of Mathematics, Louisiana State
University, Baton Rouge, LA 70803-4918, USA.\\
E-mail: nvempati@lsu.edu
\end{document}